\documentclass{amsart}

\usepackage{amssymb}
\usepackage{enumitem}
\usepackage{hyperref}
\usepackage{graphicx}
\usepackage{float}

\newtheorem{theorem}{Theorem}
\newtheorem{corollary}[theorem]{Corollary}
\newtheorem{lemma}[theorem]{Lemma}
\newtheorem{proposition}[theorem]{Proposition}

\theoremstyle{definition}

\newtheorem{example}[theorem]{Example}
\newtheorem{conjecture}[theorem]{Conjecture}

\theoremstyle{remark}
\newtheorem{remark}[theorem]{Remark}

\begin{document}
\title[EC, FLT, and powerful numbers]{Eisenstein's criterion, Fermat's last theorem, and a conjecture on powerful numbers}
\author[P.~Paparella]{Pietro Paparella}
\address{University of Washington Bothell, 18115 Campus Way NE, Bothell, WA 98011}
\email{pietrop@uw.edu}
\urladdr{http://faculty.washington.edu/pietrop/}

\begin{abstract}
Given integers $\ell > m >0$, monic polynomials $X_n$, $Y_n$, and $Z_n$ are given with the property that the complex number $\mu$ is a zero of $X_n$ if and only if the triple $(\mu,\mu+m,\mu+\ell)$ satisfies $x^n + y^n = z^n$. It is shown that the irreducibility of these polynomials implies Fermat's last theorem. It is also demonstrated, in a precise asymptotic sense, that for a majority of cases, these polynomials are irreducible via application of Eisenstein's criterion. We conclude by offering a conjecture on powerful numbers.
\end{abstract}

\maketitle

\section{Introduction}

In its original form, \emph{Fermat's last theorem} (FLT) asserts that there are no positive solutions to the Diophantine equation
\begin{equation}
x^n + y^n = z^n \label{flt}
\end{equation}
if $n>2$. As is well-known, Wiles \cite{w1995}, with the assistance of Taylor \cite{tw1995}, gave the first complete proof of FLT. 

Given integers $\ell > m >0$, we consider monic polynomials $X_n$, $Y_n$, and $Z_n$ with the property that $\mu$ is a zero of $X_n$ if and only if $(\mu,\mu+m,\mu+\ell)$ satisfies \eqref{flt}. It is shown, in a precise asymptotic sense, that for a vast majority of cases, these polynomials are irreducible via direct application of Eisenstein's criterion. Although the results fall far short of constituting a full proof of FLT -- in fact, the possibility is left open that there are infinitely-many cases to consider -- they are nevertheless appealing given that: (i) they are elementary in nature; (ii) they apply to all values of $n$ (including $n=2$); and (iii) they apply to the well-known \emph{first-case} and \emph{second-case} of \eqref{flt}. A Goldbach-type conjecture on powerful numbers is also offered.

\section{The Auxiliary Polynomials}

For fixed integers $\ell > m > 0$, let 
\begin{align}
X_n (t) = X_n (t, (\ell,m)) &:= t^n - \sum_{k=1}^{n} \binom{n}{k} t^{n - k} (\ell - m) Q_k (\ell,m),		\label{xpoly}   \\
Y_n(t) = Y_n (t, (\ell,m))&:= t^n + \sum_{k=1}^{n} (-1)^k \binom{n}{k} t^{n - k} \ell Q_k (m, m-\ell),		\label{ypoly}
\end{align}
and
\begin{equation}
Z_n(t) = Z_n (t, (\ell,m)):= t^n + \sum_{k=1}^{n} (-1)^k \binom{n}{k} t^{n - k} \left( \ell^k + (\ell-m)^k \right), \label{zpoly}
\end{equation}
where  
\begin{equation}
Q_k(\ell,m) : = \frac{\ell^k - m^k}{\ell-m} = \sum_{i=0}^{k-1} \ell^{k - 1 - i} m^i,~k=1,\dots,n. \label{qpolys}
\end{equation}

\begin{proposition}
\label{zerosflt}
If $\mu \in \mathbb{C}$, then $(\mu,\mu+m,\mu+\ell) \in \mathbb{C}^3$ satisfies \eqref{flt} if and only if $X_n (\mu) = Y_n(\mu+m) = Z_n(\mu+\ell) = 0$. 
\end{proposition}

\begin{proof}
Following the binomial theorem, notice that 
\begin{align*}
&\mu^n + (\mu+m)^n = (\mu  + \ell)^n                                                                \\
\Longleftrightarrow &\mu^n - \sum_{k=1}^{n} \binom{n}{k} \mu^{n - k} (\ell^k - m^k) = 0             \\
\Longleftrightarrow &\mu^n - \sum_{k=1}^{n} \binom{n}{k} \mu^{n - k} (\ell - m) Q_k (\ell,m) = 0    \\
\Longleftrightarrow &X_n (\mu) = 0.
\end{align*}

If $\nu := \mu+m$, then
\begin{align*}
&(\nu - m)^n + \nu^n = \left(\nu + (\ell - m)\right)^n                                                                   \\
\Longleftrightarrow &\nu^n + \sum_{k=1}^{n} \binom{n}{k} \nu^{n - k} \left((-m)^k - (\ell-m)^k \right) = 0              \\
\Longleftrightarrow &\nu^n + \sum_{k=1}^{n} (-1)^k \binom{n}{k} \nu^{n - k} \left(m^k - (-1)^k(\ell-m)^k \right) = 0    \\
\Longleftrightarrow &\nu^n + \sum_{k=1}^{n} (-1)^k \binom{n}{k} \nu^{n - k} \left(m^k - (m - \ell)^k \right) = 0        \\
\Longleftrightarrow &\nu^n + \sum_{k=1}^{n} (-1)^k \binom{n}{k} \nu^{n - k} \ell Q_k (m,m-\ell)                         \\
\Longleftrightarrow &Y_n (\nu) = Y_n(\mu+m) = 0.
\end{align*}

If $\xi := \mu + \ell$, then
\begin{align*}
&(\xi- \ell)^n + \left(\xi + (m -\ell)\right)^n = \xi^n                                                                  \\
\Longleftrightarrow &\xi^n + \sum_{k=1}^{n} \binom{n}{k} \xi^{n - k} \left((-\ell)^k + (m-\ell)^k\right) = 0            \\
\Longleftrightarrow &\xi^n + \sum_{k=1}^{n} (-1)^k \binom{n}{k} \xi^{n - k} \left(\ell^k + (-1)^k(m-\ell)^k\right) = 0  \\
\Longleftrightarrow &\xi^n + \sum_{k=1}^{n} (-1)^k \binom{n}{k} \xi^{n - k} \left(\ell^k + (\ell - m)^k\right) = 0      \\
\Longleftrightarrow &Z_n (\xi) = Z_n(\mu + \ell) = 0,
\end{align*}
and the result is established.
\end{proof}

It can be shown that if $(x,y,z) \in \mathbb{N}^3$ satisfies \eqref{flt}, with $x < y < z$, $\gcd(x,y,z)=1$, and $(\ell,m):= (z - x,y-x)$, then $\gcd(\ell,m) = 1$ \cite[p~2]{r1999}. Herein it is assumed that $\gcd{(\ell,m)}=1$.

Recall that a polynomial $f$ with coefficients from $\mathbb{Z}$ is called \emph{reducible (over $\mathbb{Z}$)} if $f=gh$, where $g$ and $h$ are polynomials of positive degree with coefficients from $\mathbb{Z}$. If $f$ is not reducible, then $f$ is called \emph{irreducible (over $\mathbb{Z}$)}.

\begin{proposition}
\label{prop:suffice}
The polynomials $X_n$, $Y_n$, and $Z_n$ are simultaneously irreducible or reducible.
\end{proposition}

\begin{proof}
Following Proposition \ref{zerosflt}, notice that 
\begin{align*}
X_n(\mu-m) = 0 
&\Longleftrightarrow (\mu-m)^n + \mu^n = (\mu - m + \ell)^n \\
&\Longleftrightarrow (\mu-m)^n + \mu^n = (\mu + \ell - m)^n \\
&\Longleftrightarrow Y_n(\mu)=0.
\end{align*}   
Thus, 
\begin{equation*}
X_n(t -m) = \prod_{\{ \mu \in \mathbb{C} \mid Y_n(\mu) = 0\}} (t - \mu) = Y_n(t).
\end{equation*}
A similar argument demonstrates that $X_n(t-\ell) = Z_n(t)$. Thus, the polynomials $X_n$, $Y_n$, and $Z_n$ are simultaneously irreducible or reducible. 
\end{proof}  

Given  
\begin{equation}
f(t) = t^n - \sum_{i=1}^n a_i t^{n-i} \in \mathbb{C}[t], \label{polyf}
\end{equation}
let
\begin{equation}
f_k (t) := t^k - \sum_{i=1}^k a_i t^{k-i},~0\leq k \leq n, \label{polyfk}    
\end{equation}
where the sum on the right is defined to be zero whenever it is empty. Notice that $f = f_n$, $f(t) = t f_{n-1}(t) - a_n$, and \( (f_j)_k = f_k\), \(0 \le k \le j\).  

\begin{lemma}
\label{polyshift}
If $f$ is the polynomial defined in \eqref{polyf}, $f_k$ is the polynomial defined in \eqref{polyfk}, and $r \in \mathbb{C}$, then 
\begin{equation*}
    f(t) = (t-r) \sum_{k=0}^{n-1} f_k(r) t^{n-1-k} + f(r). 
\end{equation*} 
\end{lemma}

\begin{proof}
Proceed by induction on $n$. If $n=1$, then 
\begin{equation*}
f(t) = t - a_1 = t - r + r - a_1 = (t-r) + f(r),
\end{equation*}
and the base-case is established.

Assume that the result holds for every polynomial of degree $j$, where $j \geq 1$. If $f(t) =  t^{j+1} - \sum_{i=1}^{j+1} a_i t^{j+1-i}$, and $r\in\mathbb{C}$, then 
\begin{align*}
f(t)    
&= t f_j(t) - a_{j+1}                                                           \\
&= t\left((t-r)\sum_{k=0}^{j-1} (f_j)_k(r) t^{j-1-k} + f_j(r) \right) - a_{j+1}       \\
&= (t-r)\sum_{k=0}^{j-1} f_k(r) t^{j-k} + tf_j(r) - a_{j+1}                   \\
&= (t-r)\sum_{k=0}^{j-1} f_k(r) t^{j-k} + (t - r)f_j(r) + rf_j(r) - a_{j+1}   \\
&= (t-r)\sum_{k=0}^{j} f_k(r) t^{j-k} + f(r),                                  
\end{align*}
establishing the result when $n= j+1$. The entire result now follows by the principle of mathematical induction.
\end{proof}

\begin{remark}
If $f(t) = t^n - \sum_{i=1}^n a_i t^{n-i} \in \mathbb{Z}[t]$ and $r \in \mathbb{Z}$ is a zero, then, following Lemma \ref{polyshift}, 
\begin{equation*}
    f(t) = (t-r) \sum_{k=0}^{n-1} f_k(r) t^{n-1-k}, 
\end{equation*}
i.e., $f$ is reducible over \( \mathbb{Z} \).
\end{remark}

The connection to FLT is now apparent from the following result.

\begin{corollary}
\label{dm}
If \( ( x, x + m, x + \ell) \in \mathbb{N}^3 \) satisfies \eqref{flt}, then \(X_n\), \(Y_n\), and \(Z_n\) are reducible. 
\end{corollary}

\begin{remark}
Corollary \ref{dm} provides a direct method of proving FLT; indeed, if it can be shown that any of the polynomials $X_n$, $Y_n$, or $Z_n$ is irreducible, then there is no solution to \eqref{flt} of the form \( ( x, x + m, x + \ell) \in \mathbb{N}^3 \). 
\end{remark}

\section{Main Results}

The following result is the most well-known irreducibility test (see, e.g., Prasolov {\cite[Theorem 2.1.3]{p2010}}) and follows from a result due to Sch\"{o}nemann (Cox \cite{c2011}). 

\begin{theorem}[Eisenstein's criterion]
\label{thm:ec}
Let $f(t) = \sum_{k=0}^n a_k t^{n-k} \in \mathbb{Z}[t]$. If there is a prime number $p$ such that:
\begin{enumerate}
\item $p \nmid a_0$;
\item $p \mid a_k$, $k =1,\dots, n$; and
\item $p^2 \nmid a_n$,
\end{enumerate}
then $f$ is irreducible over $\mathbb{Z}$.
\end{theorem}

With $f$ and $p$ as in Theorem \ref{thm:ec}, let
\begin{equation*}
\text{Eis}(f,p) := 
\begin{cases}
1, & \text{(i), (ii), and (iii) are satisfied;} \\
0, & \text{otherwise}.
\end{cases}
\end{equation*}

The following result is well-known in the literature on FLT (see Ribenboim \cite[(3B)(5), p.~81]{r1999} and references therein). For completeness, we include a proof that depends only on the definition of the polynomial $Q$ in \eqref{qpolys}. 

\begin{lemma}
\label{lem:qpoly}
Let $n >1$ and $p$ be a prime. If $\gcd(\ell,m) = 1$, $p \nmid n$, and $p \mid (\ell - m)$, then $p \nmid Q_n(\ell,m)$.
\end{lemma}

\begin{proof}
If $p \mid \ell - m$, then there is an integer $j$ such that $\ell = m + p j$. Thus, 
\begin{align*}
Q_n (\ell,m) 
&= \frac{(m+pj)^n - m^n}{pj} 								\\
&= \left( \sum_{k=0}^n \binom{n}{k} m^{n-k} p^k j^k - m^n \right)/{pj} 	\\
&= \sum_{k=1}^n \binom{n}{k} m^{n-k} p^{k-1} j^{k-1} 				\\
&= nm^{n-1} + \sum_{k=2}^n \binom{n}{k} m^{n-k} p^{k-1} j^{k-1} 		\\
&= nm^{n-1} + \sum_{k=1}^n \binom{n}{k+1} m^{n-k-1} p^{k} j^{k},
\end{align*}
and $Q_n(\ell,m) \equiv (nm^{n-1}) \not \equiv 0$ (mod $p$).
\end{proof}

If \(k\) is an integer and \(p\) is a prime, then we say that \(k\) is \emph{singly divisible by \(p\)}, denoted by \(p \mid\mid k\), whenever \(p \mid k\), but \(p^2 \nmid k \). An integer that is singly divisible by two is called \emph{singly even}. 

\begin{theorem}
\label{thm:main}
Let $X_n$ be defined as in \eqref{xpoly}. If there is a prime $p$ such that $p\mid \mid \ell - m$ and $p \nmid n$, then $X_n$ is irreducible. 
\end{theorem}

\begin{proof}
Immediate in view of \eqref{xpoly}, Theorem \ref{thm:ec}, and Lemma \ref{lem:qpoly}.
\end{proof}

\begin{remark}
The import of Theorem \ref{thm:main} is amplified by the following observation: a positive integer $a$ is called \emph{powerful} if $p^2$ divides $a$ for every prime $p$ that divides $a$; otherwise, it is called \emph{nonpowerful}. 

Golomb \cite{g1970} proved that if $\kappa(t)$ denotes the number of powerful numbers in the interval $[1,t]$, then
\begin{equation}
ct^{1/2} - 3t^{1/3} \leq \kappa(t) \leq c t^{1/2}, \label{golomb}
\end{equation}
where $c := \zeta(3/2)/\zeta(3) \approx 2.1733$ and $\zeta$ denotes the Riemann zeta function (Mincu and Panaitopol \cite{mp2009} give an improvement of \eqref{golomb}). Consequently, $\kappa(t)/t \longrightarrow 0$ as $t \longrightarrow \infty$.

If 
\begin{equation*} 
\Delta(t) := \{ \delta = \ell - m \in \mathbb{N}\mid 1 \leq m < \ell \leq t,~\delta~\text{powerful},~\gcd{(\ell,m)}=1 \},
\end{equation*}
then $|\Delta(t)| = \kappa(t)$. Thus, $|\Delta(t)|/t \longrightarrow 0$ as $t \longrightarrow \infty$.   
\end{remark}

In case $\ell-m$ is powerful, we offer the following results. 

\begin{theorem}
\label{thm:main2}
Let $Y_n$ be defined as in \eqref{ypoly}. If there is a prime $p$ such that $p \mid \mid \ell$ and $p \nmid n$, then $Y_n$ is irreducible.
\end{theorem}

\begin{proof}
Immediate in view of \eqref{ypoly}, Theorem \ref{thm:ec}, and Lemma \ref{lem:qpoly}.
\end{proof}

\begin{theorem}
\label{thm:main3}
Let $Z_n$ be defined as in \eqref{zpoly}. If $2\ell -m$ is singly even, then $Z_n$ is irreducible.
\end{theorem}

\begin{proof}
If $2\ell - m$ is singly even, then $\ell$ is odd, $m$ is even, and there is an odd integer $q$ such that $2\ell -m =2q$. As a consequence, $m = 2(\ell - q) \equiv 0$ (mod $4$). As $\ell$ and $\ell-m$ are odd, notice that 
\begin{equation*}
\left( \ell^k + (\ell-m)^k \right) \equiv 0 \pmod 2,~k=1,\dots,n.
\end{equation*}   
Moreover, since
\begin{equation*}
\ell^n + (\ell-m)^n = 2\ell^n + \sum_{k=1}^n (-1)^k \binom{n}{k} \ell^{n-k} m^k  
\end{equation*}
it follows that $\left( \ell^n + (\ell-m)^n \right)  \equiv 2\ell^n \not \equiv 0$ (mod $4$), i.e., $\text{Eis}(Z_n,2) = 1$. 
\end{proof}

\begin{example}
If $(\ell,m) = (9,4)$, $n\geq2$, $n \not \equiv 0 \pmod 5$, then $\text{Eis}(X_n,5)=1$; otherwise, if $n \equiv 0 \pmod 5$, then $\text{Eis}(Z_n,2) = 1$ since $2(9)-4 = 14$ is singly even. 
\end{example}

\begin{example}
For every positive integer \(n\), let \( P(n) := \{ (\ell,m) \in \mathbb{R}^2 \mid 0 \le m < \ell \le n,~\gcd{(\ell,m)}=1 \}\). Figure \ref{fig:1} depicts the set \( P(500) \) (please note the unorthodox position of the positive quadrant). 

Let \(p\) be a prime less than 500. If \( p \mid\mid  \ell-m\), then \( X_n (t) \) is irreducible for every positive integer \(n\) satisfying \(n \ge 500\), i.e., there is no solution to \( x^n + y^n = z^n\) of the form \((x,x+m,x+\ell) \in \mathbb{N}^3 \). Figure \ref{fig:2} contains the remaining elements of \( P(500)\) that can not be eliminated in this manner (i.e., following Theorem \ref{thm:main}). Figure \ref{fig:3} contains all pairs that can not be eliminated from Theorems \ref{thm:main} and \ref{thm:main2}. Finally, Figure \ref{fig:4} contains all elements that can not be eliminated from Theorems \ref{thm:main}, \ref{thm:main2}, and \ref{thm:main3}.

\begin{figure}[H]
\includegraphics[width=.65\linewidth]{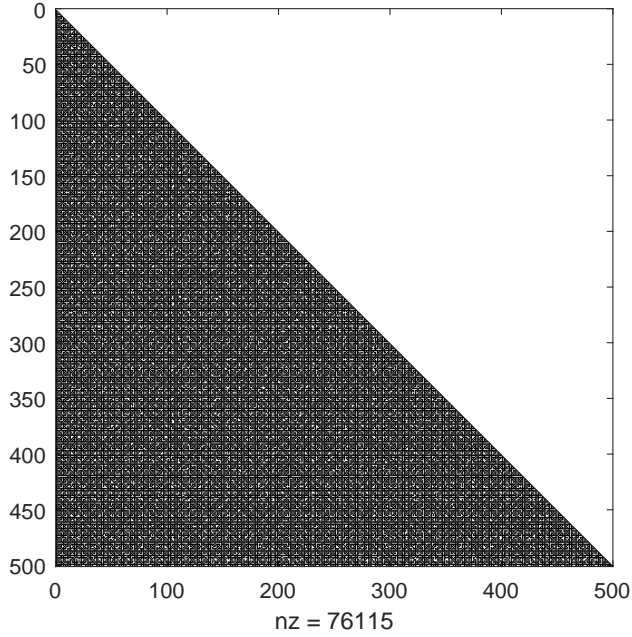}
\caption{The set \( P(500)\). Note that `nz' stands for `nonzero'.}
\label{fig:1}
\end{figure}

\begin{figure}[H]
\includegraphics[width=.65\linewidth]{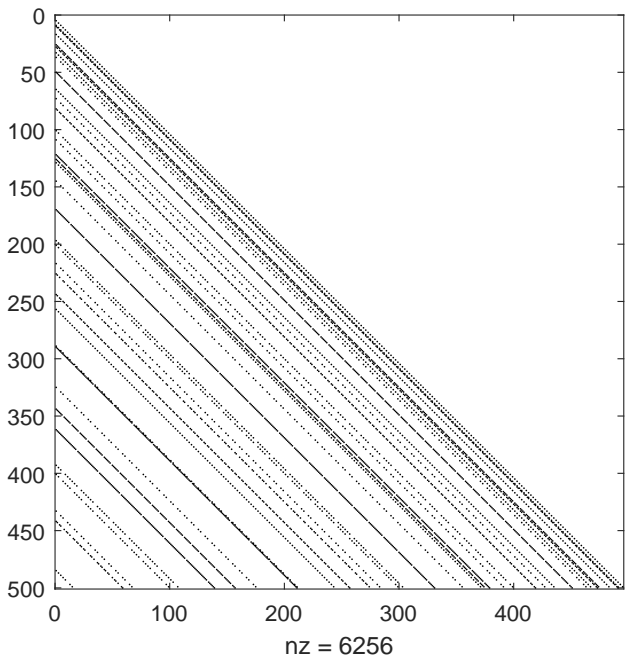}
\caption{Remaining pairs after Theorem \ref{thm:main} applied.}
\label{fig:2}
\end{figure}

\begin{figure}[H]
\includegraphics[width=.65\linewidth]{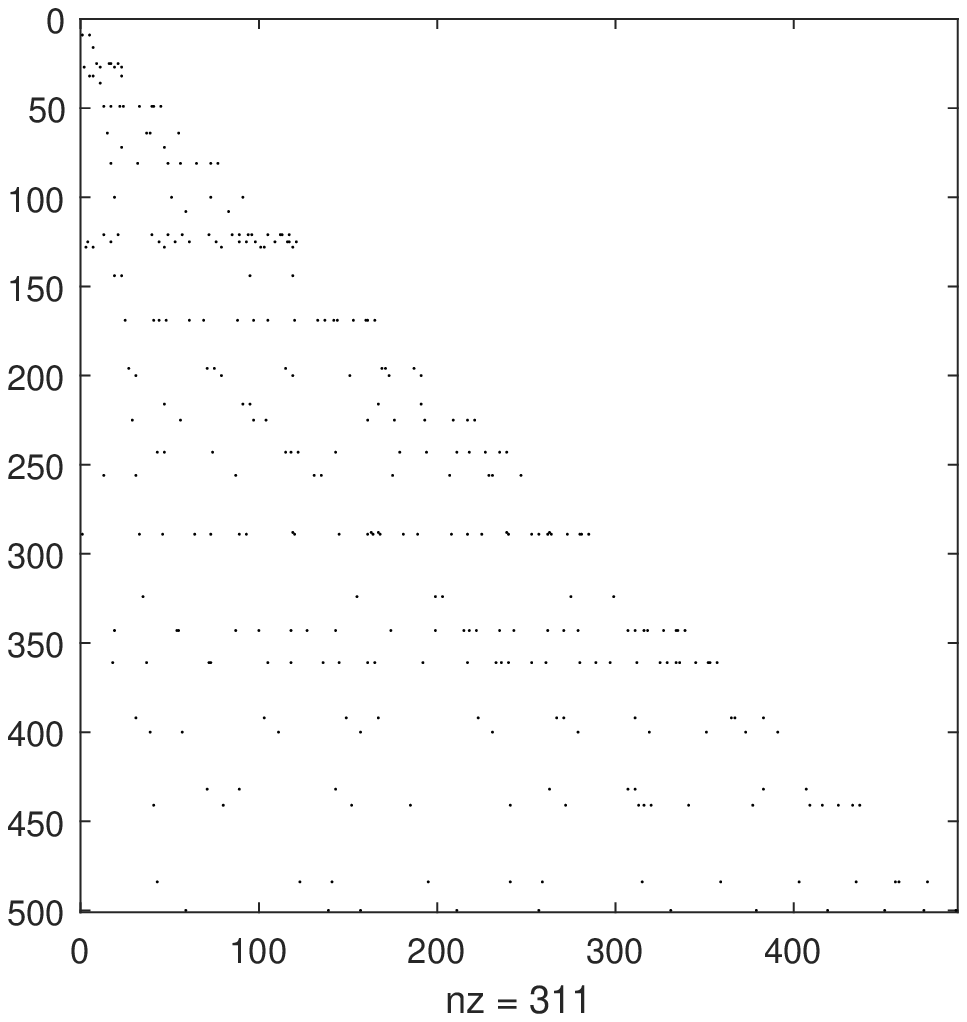}
\caption{Remaining pairs after Theorems \ref{thm:main} and \ref{thm:main2} are applied.}
\label{fig:3}
\end{figure}

\begin{figure}[H]
\includegraphics[width=.65\linewidth]{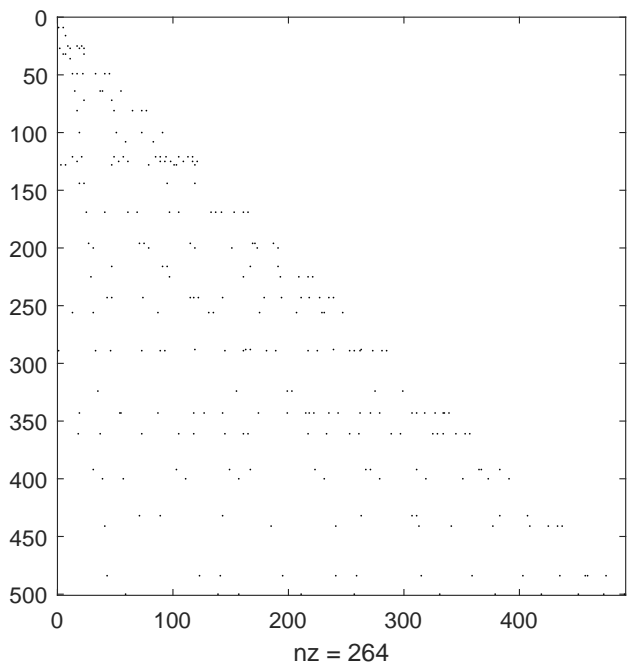}
\caption{Remaining pairs after Theorems \ref{thm:main}, \ref{thm:main2}, and \ref{thm:main3} are applied.}
\label{fig:4}
\end{figure}
\end{example}

\begin{example}
\label{badcase}
If $(\ell,m)=(9,5)$, then the irreducibility of the auxiliary polynomials cannot be asserted from the previous results.
\end{example}

As mentioned in the introduction, the above results leave the possibility that there are infinitely-many cases to resolve. The following conjecture, which generalizes Example \ref{badcase}, would not only establish this, but is seemingly of great import in and of itself and could be as difficult to resolve as the Goldbach conjecture (\cite{p2017}). 

\begin{conjecture}
If $a>1$ is powerful, then there is a prime $p$ and a powerful number $b$ such that $a=b+p$.
\end{conjecture}

\bibliographystyle{amsplain}
\bibliography{ecflt}

\end{document}